\documentclass[11pt]{amsart}
\usepackage{graphicx} 
\usepackage{hyperref,amsfonts,amsmath,amssymb,amsthm}
\usepackage{amscd,mathrsfs}
\usepackage{mathtools}
\usepackage{comment}
\usepackage{stmaryrd}

\numberwithin{equation}{section}

\theoremstyle{plain}
\newtheorem{theorem}{Theorem}[section]

\newtheorem{lemma}[theorem]{Lemma}

\theoremstyle{definition}
\newtheorem{definition}[theorem]{Definition}

\theoremstyle{remark}

\DeclareMathOperator{\re}{Re}
\DeclareMathOperator{\im}{Im}

\DeclareMathOperator{\spt}{spt}
\DeclareMathOperator{\reg}{reg}
\DeclareMathOperator{\sing}{sing}

\title[Stationary integral currents]
{Stationary integral cycles asymptotic to \\
a pair of oriented $n$-planes in $\mathbb{R}^{2n}$}
\author{Jooho Lee}
\address{School of Mathematics, Korea Institute for Advanced Study\\
                85 Hoegiro, Dongdaemun-gu\\
                Seoul 02455, Republic of Korea}
\email{jooholee@kias.re.kr}

\thanks{2010 {\em Mathematics Subject Classification.} 53A10.}

\begin{document}

\begin{abstract}
    We prove a uniqueness theorem for stationary integral $n$-cycles asymptotic to a pair of $n$-planes satisfying the symmetric angle condition. For the general angle condition, we obtain a classification for complete smooth embedded oriented minimal submanifolds with vanishing first de Rham cohomology that are asymptotic to a pair of $n$-planes.
\end{abstract}

\maketitle

\section{Introduction}
In \cite{HS85}, Hardt and Simon showed that given an area-minimizing hypercone in $\mathbb{R}^{n+1}$, there are complete smooth area-minimizing hypersurfaces asymptotic to the cone.

In codimension $1$, quadratic cones $C^{p,q}$ are the simplest area-minimizing cones with an isolated singularity, 
\begin{equation*}
    C^{p,q}=\{(x,y)\in \mathbb{R}^{p+1}\times \mathbb{R}^{q+1} : q|x|^2=p|y|^2\},
\end{equation*}
where $p,q$ are positive integers and $p+q=n-1$. When $n=7$, the cones $C^{3,3}$ and $C^{2,4}$ are area-minimizing. When $n>7$, every $C^{p,q}$ is area-minimizing. 

In \cite{SS86}, Simon and Solomon proved that the only minimal hypersurfaces asymptotic to an area-minimizing quadratic cone are either the cone or a leaf of foliation up to similarities of $\mathbb{R}^{n+1}$.

One can ask the following questions:

\textit{
Given an area-minimizing cone $C$ with an isolated singularity at $0$, suppose that a stationary integral current $T$ is asymptotic to $C$ near infinity. What can we say about uniqueness and regularity of $T$?
}

Let $P$ and $Q$ be oriented $n$-planes in $\mathbb{R}^{2n}$. Then, there exists an oriented orthonormal basis $e_1, \cdots, e_{2n}$ for $\mathbb{R}^{2n}$ such that 
        \begin{align*}
            P&=e_1 \wedge \cdots \wedge e_n, \\
            Q&=(\cos{\theta_1}e_1 + \sin{\theta_1}e_{n+1}) \wedge \cdots \wedge (\cos{\theta_n}e_n + \sin{\theta_n}e_{2n})
        \end{align*}
        where $0 \le \theta_1 \le \cdots \le \theta_{n-1} \le \pi/2$ and $\theta_{n-1} \le \theta_{n} \le \pi - \theta_{n-1}$. These angles are called the characterizing angles of $P$ and $Q$.

The angle criterion~\cite{Nance87,L89} says that a pair of oriented $n$-planes is area-minimizing if and only if 
\[
    \theta_{n} \le \theta_1 + \cdots +\theta_{n-1}.
\]

For each pair of oriented $n$-planes $P$ and $Q$ satisfying the angle condition
\begin{equation}\label{eq:angle}
    \theta_{n} = \theta_1 + \cdots +\theta_{n-1},
\end{equation} there is a smooth area-minimizing submanifold asymptotic to $P+Q$. This submanifold is known as a Lawlor neck. 

Consider two $n$-planes
\begin{align*}
    P=\{(x,y) \in \mathbb{R}^{2n} : y=0\}, \quad
    Q=\{(x,y) \in \mathbb{R}^{2n} : y=\tan(\pi/n)x\}.
\end{align*}

Then, with respect to the basis 
\begin{align*}
    e_1=\frac{\partial}{\partial x^1}, \cdots, e_n=\frac{\partial}{\partial x^n}, \\
    e_{n+1}=-\frac{\partial}{\partial y^1}, \cdots, e_{2n}=-\frac{\partial}{\partial y^n},
\end{align*}
the characterizing angles are 
\begin{equation}\label{eq:symmetric angle}
    \theta_1 = \cdots = \theta_{n-1} = \frac{\pi}{n}, \quad 
    \theta_{n} = \frac{n-1}{n}\pi.
\end{equation}

In this case, the Lawlor neck asymptotic to $P+Q$ is known as the Lagrangian catenoid and has the following explicit expression:  
\begin{equation*}
    M=\{(x,y) \in \mathbb{R}^{2n} : |y|x=|x|y, \, \operatorname{Im}(|x|+\sqrt{-1}|y|)^n=1, \, |y|<|x|\tan{\frac{\pi}{n}} \}.
\end{equation*}
Moreover, $M$ is invariant under the $SO(n)$-diagonal action on $\mathbb{R}^{2n}$. 

In \cite{IJO2016}, Imagi, Joyce and dos Santos proved a uniqueness theorem for smooth exact special Lagrangian submanifolds asymptotic to a pair of $n$-planes. In \cite{Li-Székelyhidi}, Li and Székelyhidi extended the uniqueness result to Lagrangian integral $n$-cycles under a smooth exact approximation condition.

\subsection{Main results}

We now state the main theorems of this paper. First, we prove a uniqueness theorem for stationary integral $n$-cycles asymptotic to a pair of $n$-planes satisfying the symmetric angle condition \eqref{eq:symmetric angle}:
\begin{theorem}\label{main1}
    For $n \ge 2$, let $P$ and $Q$ be a pair of oriented $n$-planes in $\mathbb{R}^{2n}$ satisfying the symmetric angle condition \eqref{eq:symmetric angle}. If a stationary integral $n$-cycle $T$ is asymptotic to $P+Q$ in the sense of currents and varifolds, then, up to translation and scaling, $T$ is either $P+Q$ or the Lagrangian catenoid. 
\end{theorem}

Second, we classify complete smooth embedded oriented minimal submanifolds with vanishing first de Rham cohomology that are asymptotic to a pair of $n$-planes satisfying the angle condition \eqref{eq:angle}:
\begin{theorem}\label{main2}
    For $n\ge 3$, let $M$ be a complete smooth embedded oriented minimal submanifold of dimension $n$ in $\mathbb{R}^{2n}$ with $H^1(M;\mathbb{R})=0$. Suppose that $M$ is asymptotic to $P+Q$ in the sense of currents and varifolds, where $P$ and $Q$ intersect transversely and satisfy the angle condition \eqref{eq:angle}. Then, up to translation and scaling, $M$ is a Lawlor neck.
\end{theorem}

\section*{Acknowledgments}
I am deeply grateful to Rick Schoen for suggesting this problem and for the many enlightening discussions. I also thank Kyeongsu Choi for insightful discussions and encouragement. 
The author was supported in part by the KIAS Individual Grant MG104801.

\section{Preliminaries}

\subsection{Integral currents}
We briefly recall some basic facts about stationary integral currents. For further background, we refer to \cite{Almgren,Federer-GMT,Simon-GMT,Morgan-GMT}.

A $n$-dimensional rectifiable current
\[
T=t(S,\theta,\xi)
\]
in $\mathbb{R}^N$ consists of
\begin{itemize}
    \item an underlying $n$-rectifiable set $S$,
    \item a positive integer-valued multiplicity function
    \[
    \theta:S\to \mathbb{Z}^+,
    \]
    \item a unit simple $n$-vector-valued orientation function
    \[
    \xi:S\to \bigwedge_n \mathbb{R}^N.
    \]
\end{itemize}
A rectifiable $n$-current $T$ is called an integral current if $\partial T$ is a rectifiable $(n-1)$-current. In particular, an integral current $T$ with $\partial T=0$ is called an integral cycle.
The mass $\mathbb{M}(T)$ of $T$ is 
\[
\mathbb{M}(T)=\int_S \theta\,d\mathcal{H}^k.
\]

An integral $n$-current $T$ is called stationary if
\[
    \delta T(X)=\int_S \operatorname{div}_{T} X \, d\|T\|=0
\]
for every compactly supported smooth vector field $X$ in $\mathbb{R}^N$.

\subsection{Calibrations}

We recall some basic facts about calibrations and calibrated currents. For further details on calibrations, see \cite{HL,Harvey}.

A closed $n$-form $\phi$ on $\mathbb{R}^N$ is called a calibration if 
\[
    \phi(\xi) \le 1
\]
for all unit simple $n$-vector $\xi$. 
An integral $n$-current $T=t(S,\theta,\xi)$ is said to be calibrated by $\phi$ if 
\[
    \phi_x(\xi(x))=1
\]
for $\|T\|$-almost every $x$.

Let $\mathbb{C}^n$ have the standard coordinates $z^j=x^j+iy^j$, for $j=1, \cdots, n$. We consider the standard complex structure $J$, K{\"a}hler form $\omega$, and holomorphic volume form $\Omega$. In the standard coordinates,
        \begin{align*}\label{eq:Calabi-Yau}
            J\left(\frac{\partial}{\partial x^j} \right)=\frac{\partial}{\partial y^j}, \quad J\left(\frac{\partial}{\partial y^j} \right)=-\frac{\partial}{\partial x^j}\\
            \omega = \sum_{j=1}^{m} dx^j \wedge dy^j = \frac{i}{2} \sum_{j=1}^{n} dz^j \wedge d\bar{z}^j \\
            \Omega = dz^{1} \wedge \cdots \wedge dz^{n}.          
        \end{align*}

For each $1\le k \le n$, the form $\omega^k/k!$ is a calibration, called the K\"ahler calibration. A locally integral $2k$-current $T$ is calibrated by $\omega^k/k!$ if and only if $\Vec{T}$ is invariant under $J$.

For each real number $\theta$, the real part $\re(e^{-i\theta}\Omega)$ is a calibration, called the special Lagrangian calibration of phase $\theta$. A locally integral $n$-current $T$ is said to be special Lagrangian with phase $\theta$ if $T$ is calibrated by $\re(e^{-i\theta}\Omega)$.

\section{Stationary integral currents}
Throughout this paper, $T$ denotes a stationary integral $n$-cycle of multiplicity $1$ in $\mathbb{R}^N$.
\subsection{Partial regularity}
In this section, we prove a partial regularity for stationary integral $n$-cycle in $\mathbb{R}^N$.

We recall the stratification theorem for stationary integral currents \cite{Federer70, Almgren83, White97}.
Let $T$ be a stationary integral n-current in $\mathbb{R}^N$.
We denote by $\Sigma_k$ the set of $x\in \spt T$ such that no tangent cone at $T$ at $x$ has more than a k-dimensional space of translational symmetries, i.e. 
\begin{align*}
    \Sigma_k=\{x\in \spt T: \, \text{no tangent cone has the form}\, \mathbb{R}^{k+1}\times C
\}.
\end{align*}
Then, we have a sequence of increasing sets
\begin{equation*}
    \Sigma_0 \subseteq \Sigma_1 \subseteq \cdots \subseteq \Sigma_n = \sing T \subseteq \spt T.
\end{equation*}
Then, Each $\Sigma_k$ has Hausdorff dimension at most $k$, and $\Sigma_0$ is countable.

Applying the stratification theorem for stationary integral currents, we have the partial regularity.

\begin{theorem}\label{reg} 
Let $T$ a stationary integral n-cycle in $\mathbb R^N$. If the density of $T$ at infinity at most $2$, then either it is a cone or it is a smooth embedded submanifold away from a closed set $\Sigma$ of Hausdorff dimension at most $n-3$. Moreover, the support of $T$ is connected unless $T$ is a disjoint union of two $n$-planes.
\end{theorem}

    \begin{proof}
    If $\Theta(T, x)=2$ for some $x\in \spt T$, by monotonicity formula, $T$ is a cone at $x$. Now suppose that $\Theta(T,x)<2$ for all $x\in \spt T$.
    
     We will show $\Sigma_n=\Sigma_{n-1}=\Sigma_{n-2}$.
    Let $x \in \Sigma_n \setminus \Sigma_{n-1}$. Then, $x$ has a tangent cone of the form $\mathbb{R}^n$. Since $\Theta(T,x)<2$, the tangent cone must be multiplicity 1 subspace, therefore $x$ is a regular point. This is a contradiction.

    Next, $x\in \Sigma_{n-1} \setminus \Sigma_{n-2}$. Then, $x$ has a tangent cone $C_1$ of the form $\mathbb{R}^{n-1}\times C_{1}'$, where $C_{1}$ is a 1-dimensional minimal cone with an isolated singularity. Since a 1-dimensional cone is a union of rays and $\Theta(T,x)<2$, $C_1$ is a union of three half-planes. On the other hand, a union of three half-planes always has non-empty boundary regardless of orientation. Thus, $C_1$ cannot be a tangent cone of integral cycle. This is a contradiction.

    Let $x\in \Sigma_{n-2} \setminus \Sigma_{n-3}$. Then, $x$ has a tangent cone $C_2$ of the form $\mathbb{R}^{n-1}\times C_{2}'$, where $C_{2}$ is a 2-dimensional minimal cone with an isolated singularity. Thus, the intersection of $C_{2}'$ with the unit sphere must be a union of smooth geodesics in $S^{N-1}$. Since $\Theta(T,x)<2$, we have $C_{2}'$ is a single geodesic, therefore, $C_2$ is a multiplicity 1 subspace, therefore $x$ is a regular point. This is a contradiction.
    
      Finally, the last statement follows from the monotonicity formula.
    \end{proof}

\subsection{Regularity at infinity}
In \cite{Schoen}, Schoen introduced the notion of regularity at infinity for complete minimal hypersurfaces. This notion extends naturally to stationary integral $n$-cycles in $\mathbb{R}^N$.

\begin{definition}\label{regular at infinity}
A stationary locally integral $n$-cycle $T$
in $\mathbb{R}^N$ is said to be \emph{regular at infinity}
if there exist a compact set $K\subset\mathbb{R}^N$ and $R_0>0$
such that $\spt T\setminus K$ has finitely many
connected components, each given by the graph of a smooth map
\[
    u_i:P_i\setminus\overline{B_{R_0}^{P_i}}\longrightarrow P_i^\perp,
    \qquad i=1,\dots,k,
\]
with bounded gradient, where $P_i$ is an $n$-plane and $B_R^{P_i}=\{x\in P_i:|x|<R\}$.
Moreover, we require as $|x|\to\infty$, 
\[
    u_i(x)=
    \begin{cases}
        a_i+b_i\log|x|+w_i(x),    & n=2,\\
        a_i+b_i|x|^{2-n}+w_i(x),  & n\ge 3,
    \end{cases}
\]
where $w_i(x)=O(|x|^{1-n})$ and $a_i,b_i\in P_i^\perp$.
\end{definition}

Next, we show that sublinear growth implies regularity at infinity.

\begin{theorem}\label{inf}
    Suppose that there exist a compact set $K\subset\mathbb{R}^N$ and $R_0>0$
such that $\spt T\setminus K$ has finitely many
connected components, each given by the graph of a smooth map
\[
    u_i:P_i\setminus\overline{B_{R_0}^{P_i}}\longrightarrow P_i^\perp,
    \qquad i=1,\dots,k,
\]
where $P_i$ is an $n$-plane and $B_R^{P_i}=\{x\in P_i:|x|<R\}$.
If each $u_i$ satisfies 
    \begin{equation*}
        |f(x)|+|x||Df(x)| \le C |x|^{\alpha}
    \end{equation*}
    for some $0<\alpha <1$ and for all sufficiently large $|x|$, then $T$ is regular at infinity.
\end{theorem}

\begin{proof}
    Assume that $n\ge3$. The case $n=2$ follows by the same argument. Using the standard elliptic estimates, we also have $|x|^k|D^{k}f|\le C|x|^{\alpha}$ for all $k=0,1,2 \cdots$.
    Choosing coordinates, we may assume that
    \begin{align*}
    P=\{x \in \mathbb{R}^{N} : x^j =0 \text{ for } j=n+1, \cdots, N\}
    \end{align*}
    Then, $f$ satisfies the minimal surface system
    \begin{align*}
        g^{ij}f_{ij}=0 
    \end{align*}
where $g_{ij}=\delta_{ij}+f_i\cdot f_j$. Since $f_i=O(|x|^{\alpha-1})$ and $g^{ij}=\delta^{ij}+O(|x|^{2(\alpha -1)})$, we have
\begin{align*}
    \Delta f= (\delta^{ij}-g^{ij})f_{ij}=O(|x|^{3\alpha -4}).
\end{align*}

Recall the standard Fredholm theory for the Laplace operator (\cite{Lockhart-McOwen}, \cite{Bartnik}).
Suppose that a function $\varphi:\mathbb{R}^n\setminus B_R \to \mathbb{R}$ satisfies
\begin{align}
    |\varphi(x)|\leq c|x|^\alpha \\ |\Delta \varphi|+|x||\nabla \Delta \varphi|\leq C |x|^p 
\end{align}
for some $p<-1$ and not an integer.
If $-2<p<-1$, then 
\begin{equation}
    |\varphi(x)|+|x||\nabla \varphi(x)|\leq c|x|^{p+2}.
\end{equation}
If $-n<p<-2$, then there is a real number $a$ such that
\begin{equation}
    |\varphi(x)-a|+|x||\nabla \varphi(x)| \leq c|x|^{p+2}.
\end{equation}

If $-n-1<p<-n$, then there are real numbers $a$, $b$ so that
\begin{equation}
    |\varphi(x)-a-b|x|^{2-n}|+|x||\nabla (\varphi(x)-b|x|^{2-n})|\leq c |x|^{p+2}
\end{equation}
Similar statement holds for smaller $p$.

Since $3\alpha-4<-1$, we can choose $p$ so that $-2<p<-1$ and $3\alpha-4<p<\alpha-2$. Thus, we get the extra decay $f=O(|x|^{p+2})$. Plugging into minimal surface system, we iterate and obtain as much decay as we want.
\end{proof}

\subsection{Calibrations} 
Throughout this section, let $T$ be a stationary integral $n$-cycle in $\mathbb{R}^N$ whose singular set has Hausdorff dimension at most $n-2$. We assume that $T$ is regular at infinity.

We fix $R>R_0$ and set
\[
    F_i(x)=x+u_i(x),\qquad
    S_i=\partial B_R^{P_i},\qquad
    \Gamma_i=F_i(S_i).
\]

Let $E_i$ denote the current of integration over
$F_i(P_i\setminus\overline{B_R^{P_i}})$, with the orientation
induced by $T$. Then
\[
    T=T_R+\sum_{i=1}^k E_i,
\]
where $T_R$ is a compactly supported integral $n$-current.

Let $\eta_i$ denote the outward unit conormal to $T_R$ along
$\Gamma_i$. 
The first variation formula gives the following identities.

\begin{lemma}\label{flux}
We have
\begin{align}
    \sum_{i=1}^k \int_{\Gamma_i} \eta_i\,d\mathcal{H}^{n-1}
    &=0,\\
    \sum_{i=1}^k \int_{\Gamma_i}
    \langle X,\eta_i\rangle\,d\mathcal{H}^{n-1}
    &=n\,\mathbb{M}(T_R),
\end{align}
where $X$ is the position vector field on $\mathbb{R}^N$.
\end{lemma}

\medskip

Let $C$ be an oriented cone consisting of oriented $n$-planes $P_i$ for $i=1,\cdots, k$. We assume that the orthogonal complements $P_{i}^{\perp}$ are linearly independent, i.e., if 
\[
    b_1+\cdots+b_k=0, \quad b_i\in P_{i}^{\perp},
\]
then $b_i=0$ for all $i$. 

Let $T$ be a stationary integral $n$-cycle in $\mathbb{R}^N$ whose singular set has Hausdorff dimension at most $n-2$. We assume that $T$ is regular at infinity. In this setting, we obtain the following.

\begin{theorem}\label{calibration}
    Let $\phi$ be a constant calibration. Assume that $C$ is calibrated by $\phi$. If $T$ is asymptotic to $C$ in the the sense of currents, then $T$ is calibrated by $\phi$. Moreover, each graph satisfies $u=O(|x|^{1-n})$, up to a translation.
\end{theorem}

\begin{proof}
We first work on a single graph. Translating the plane $P$, we may assume that 
\[
u(x)=
\begin{cases}
b\log|x|+w(x),   & n=2,\\
b|x|^{2-n}+w(x), & n\ge 3,
\end{cases}
\]    
where $w(x)=O(|x|^{1-n})$. 

For $u:P \to P^{\perp}$ and $x\in P$, let $r=|x|$ and $\omega=x/|x|$.
Choose an oriented orthonormal basis $e_1, \cdots, e_{n-1}\in T_xS$.

The outward unit conormal $\eta$ is given by
\[
\eta =
\frac{
    \partial_r +u_r
    -g^{\alpha\beta}(u_r\cdot u_\alpha)(e_\beta+u_\beta)
}{
    \sqrt{
        1 +|u_r|^2
        -g^{\alpha\beta}(u_r \cdot u_\alpha)(u_r \cdot u_\beta)
        }
}.
\]
where $g_{\alpha\beta} = \delta_{\alpha\beta}+u_\alpha \cdot u_\beta$ 
for $1\le \alpha, \beta \le n-1$.

The asymptotic expansion of $\eta$ is 
\[
\eta=
\begin{cases}
\partial_r +br^{-1}+O(r^{-2}),       & n=2,\\
\partial_r +(2-n)br^{1-n}+O(r^{-n}), & n\ge 3.
\end{cases}
\]    

The flux integral is
\[
\int_{\Gamma} \eta =
\begin{cases}
|\mathbb{S}^1|b+O(r^{-1}),          & n=2,\\
(2-n)|\mathbb{S}^{n-1}|b+O(r^{-1}), & n\ge 3.
\end{cases}
\]

By lemma \ref{flux}, letting $r\to \infty$, we have $\sum b_i=0$. Since orthogonal complements $P_{i}^{\perp}$ are linearly independent, we get $b_i=0$ for all $i$.

Therefore, we have $u(x)=O(|x|^{1-n})$ and $Du(x)=O(|x|^{-n})$ for all $n\ge 2$. 

\medskip

    For $R$ large, we consider the $\phi$-calibration energy
    introduced in \cite{Miura-Rupp},
\begin{equation}\label{eq:energy}
    \mathcal{E}_{\phi}[T_R]
    = \mathbb{M}(T_R)-\langle T_R,\phi \rangle.
\end{equation}
Thus, $\mathcal{E}_{\phi}[T_R]$ is monotonically increasing in $R$ and $T$ is calibrated by $\phi$ if and only if $\mathcal{E}_{\phi}[T_R]=0$ for all large enough $R$. 

Since $\phi$ is a constant calibration, using Cartan's formula, we have $d(\iota_X \phi)=n\phi$.
Using Stokes' theorem and theorem \ref{flux}, we have

\begin{equation*}
        n\mathcal{E}_{\phi}[T_R]
        =\sum_{i=1}^k\int_{\Gamma_i} \langle X,\eta_i \rangle \,d\mathcal{H}^{k-1}-\sum_{i=1}^k\int_{\Gamma_i}\iota_X \phi,   
\end{equation*}
where $X$ is the position vector field.

Since $u(x)=O(|x|^{1-n})$ and $Du(x)=O(|x|^{-n})$, a direct computation shows

\begin{align*}
    \int_{\Gamma} \langle X,\eta \rangle \,d\sigma
    &=\int_S \langle x+u(x),\partial_r +O(|x|^{-n}) \rangle 
    \sqrt{\det(\delta_{ij}+u_i\cdot u_j)}\,d\sigma \\
    &=|\mathbb{S}^{n-1}| R^n + O(R^{1-n}).
\end{align*}
and
\begin{align*}
    \int_{\Gamma}\iota_X \phi
    &=\int_S \phi(r\omega+u(x),e_1+e_1(u), \cdots, e_{n-1}+e_{n-1}(u)) \,d\sigma \\
    &=|\mathbb{S}^{n-1}| R^n + O(R^{1-n}).
\end{align*}

Letting $R\to \infty$ in \eqref{eq:energy}, we have $E(T_R)=0$ for all sufficiently large $R>0$.
\end{proof}

\section{Special Lagrangians}

Let $P$ and $Q$ be oriented $n$-planes in $\mathbb{R}^{2n}$. Then, there exists an oriented orthonormal basis $e_1, \cdots, e_{2n}$ for $\mathbb{R}^{2n}$ such that 
        \begin{align*}
            P&=e_1 \wedge \cdots \wedge e_n, \\
            Q&=(\cos{\theta_1}e_1 + \sin{\theta_1}e_{n+1}) \wedge \cdots \wedge (\cos{\theta_n}e_n + \sin{\theta_n}e_{2n})
        \end{align*}
        where $0 \le \theta_1 \le \cdots \le \theta_{n-1} \le \pi/2$ and $\theta_{n-1} \le \theta_{n} \le \pi - \theta_{n-1}$. These angles are called the characterizing angles of $P$ and $Q$.

The angle criterion~\cite{Nance87,L89} says that a pair of oriented $n$-planes is area-minimizing if and only if 
\[
    \theta_{n} \le \theta_1 + \cdots +\theta_{n-1}.
\]

Let $\mathbb{C}^n$ have the standard coordinates $z^j=x^j+iy^j$, for $j=1, \cdots, n$. We consider the standard complex structure $J$, K{\"a}hler form $\omega$, and holomorphic volume form $\Omega$. In the standard coordinates,
        \begin{align*}
            J\left(\frac{\partial}{\partial x^j} \right)=\frac{\partial}{\partial y^j}, \quad J\left(\frac{\partial}{\partial y^j} \right)=-\frac{\partial}{\partial x^j}\\
            \omega = \sum_{j=1}^{m} dx^j \wedge dy^j = \frac{i}{2} \sum_{j=1}^{n} dz^j \wedge d\bar{z}^j \\
            \Omega = dz^{1} \wedge \cdots \wedge dz^{n}.          
        \end{align*}

For each real number $\theta$, the real part $\re(e^{-i\theta}\Omega)$ is a calibration, called the special Lagrangian calibration of phase $\theta$. A locally integral $n$-current $T$ is said to be special Lagrangian with phase $\theta$ if $T$ is calibrated by $\re(e^{-i\theta}\Omega)$.

For each pair of oriented $n$-planes $P$ and $Q$ satisfying the angle condition
\begin{equation}
    \theta_{n} = \theta_1 + \cdots +\theta_{n-1},
\end{equation} there is a special Lagrangian submanifold asymptotic to $P+Q$, which is known as a Lawlor neck \cite{L89}. 

In \cite{IJO2016}, Imagi, Joyce and dos Santos proved a uniqueness theorem for smooth exact special Lagrangian submanifolds asymptotic to a pair of $n$-planes. In \cite{Li-Székelyhidi}, Li and Székelyhidi extended the uniqueness result to Lagrangian integral $n$-cycles under a smooth exact approximation condition.

\subsection{Lagrangian Catenoid}

Consider two oriented $n$-planes
\begin{align*}
    P=\{(x,y) \in \mathbb{R}^{2n} : y=0\}, \quad
    Q=\{(x,y) \in \mathbb{R}^{2n} : y=\tan(\pi/n)x\}.
\end{align*}

Then, with respect to the basis 
\begin{align*}
    e_1=\frac{\partial}{\partial x^1}, \cdots, e_n=\frac{\partial}{\partial x^n}, \\
    e_{n+1}=-\frac{\partial}{\partial y^1}, \cdots, e_{2n}=-\frac{\partial}{\partial y^n},
\end{align*}
the characterizing angles are 
\begin{equation}
    \theta_1 = \cdots = \theta_{n-1} = \frac{\pi}{n}, \quad 
    \theta_{n} = \frac{n-1}{n}\pi.
\end{equation}

In this case, the Lawlor neck asymptotic to $P+Q$ is known as the Lagrangian catenoid and has the following explicit expression:  
\begin{equation*}
    M=\{(x,y) \in \mathbb{R}^{2n} : |y|x=|x|y, \, \operatorname{Im}(|x|+\sqrt{-1}|y|)^n=1, \, |y|<|x|\tan{\frac{\pi}{n}} \}.
\end{equation*}
Moreover, $M$ is invariant under the $SO(n)$-diagonal action on $\mathbb{R}^{2n}$.

\subsection{$SO(n)$-action in $\mathbb{R}^{2n}$}
In this section, we consider $SO(n)$-diagonal action on $\mathbb R^{2n}$. $SO(n)$ acts on $\mathbb R^{2n}$ by
\begin{align*}
    A \cdot (x,y)=(A x,Ay)
\end{align*}
for $A\in SO(n)$.

Notice that a n-dimension submanifold $M$ is invariant under $SO(n)$-diagonal action if and only if all rotation vector fields $V$ on $\mathbb{R}^{2n}$ are tangent to $M$,
\begin{equation*}
    V=x^j \frac{\partial}{\partial x^i}-x^i \frac{\partial}{\partial x^j} 
    + y^j \frac{\partial}{\partial y^i}-y^i \frac{\partial}{\partial y^j}.
\end{equation*}

There are three orbit types for this action. Let 
\begin{equation*}
    C=\{(x,y) \in \mathbb{R}^{2n} : |y|x=\pm|x|y \}.
\end{equation*} Then, the dimension of each orbit is
\begin{equation*}
           \text{dim(Orb($p$))}=   \begin{cases} 
                          0 &  p=(0,0) \\
                          n-1 &  p \in C \setminus \{0\} \\
                          2n-3 &  p \notin C 
                        \end{cases}
\end{equation*}

The $(n+1)$-dimensional cone $C$ in $\mathbb{R}^{2n}$ is a union of planes $x=\lambda y$. In particular, the cone $C$ is invariant under $SO(n)$-diagonal action.

Let $\mu_{ij}=x^iy^j-x^jy^i$ for $1 \le i, j \le n$. Then, 
$p \in C \setminus \{0\}$ if and only if there exists $i$ such that 
\begin{align*}
    x_i(p)\neq 0 \quad\text{  or  } \quad y_i(p)\neq 0,\\
    \mu_{ij}(p)=0 \quad\text{for all} \quad 1 \le  j \le n.
\end{align*}

\begin{lemma}\label{basis}
    Let $C=\{p=(x,y) \in \mathbb{R}^{2n} : |y|x=\pm|x|y \}$. Then, each $p \in C \setminus \{0\}$ there exists a basis $\{X_1, \cdots, X_n, JX_1, \cdots, JX_n\}$ for $T_{p}\mathbb{R}^{2n}$ such that 
    \begin{enumerate}
        \item $\{X_1, \cdots, X_n, JX_1\}$ is a basis for $T_{p}C$
        \item $\{JX_2, \cdots, JX_n\}$ is a basis for $N_{p}C$
        \item $JX_i\cdot X_j=0$ for all $1 \le i \le j \le n$.
    \end{enumerate}
\end{lemma}

\begin{proof}
    Without loss of generality, we may assume $x^1(p)\neq 0$. Define $X_j$ as follows:
    \begin{align*}
        X_1&=\sum_{j=1}^{n} \left(x^j\frac{\partial}{\partial x^j} + y^j\frac{\partial}{\partial y^j}\right)\\
        X_j&= x^{j}\frac{\partial}{\partial x^1} -x^{1}\frac{\partial}{\partial x^j} 
              + y^{j}\frac{\partial}{\partial y^1} -y^{1}\frac{\partial}{\partial y^j},
    \end{align*}
    for $j=2,\cdots, n$. Then, $\{X_1, \cdots, X_n\}$ are linearly indepedent subset of $T_{p}C$. Since $\mu_{1j}=0$ on $C$, $JX_j=-D(\mu_{1j})\in N_{p}C$. 
    Since $JX_1(\mu_{1j})=-JX_1 \cdot JX_j=0$, $JX_1$ is tangent to $C$. 
    For each $j=2, \cdots, n$, we have
    \begin{align*}
        JX_1 \cdot X_1 &= 0 \\
        JX_1 \cdot X_j &= 2\mu_{1j}=0 \\
        JX_1 \cdot JX_j &= X_1 \cdot X_j =0.
    \end{align*} This implies that $\{X_1, \cdots, X_n, JX_1\}$ is linearly independent.
    For $2 \le i <j \le n$, 
    \begin{align*}
        JX_i \cdot X_j &= -\mu_{ij}=0.
    \end{align*}
\end{proof}

\begin{lemma}\label{lemma:sLag}
    Let $T$ be a special Lagrangian integral current of phase $0$ in $\mathbb{C}^n$.
    Assume that $\spt T\subset C$ is connected and $\spt T\subset \{0\}$.
    If $T$ is asymptotic to $P+Q$ in the sense of currents and varifolds, then either $T$ is the Lagrangian catenoid or $T=P+Q$. 
\end{lemma}

\begin{proof}
    Consider $\varphi=\operatorname{Im}(|x|+\sqrt{-1}|y|)^n$. We prove $d\varphi=0$ using $\operatorname{Im}(\Omega)=0$. 
    On the cone
    \begin{align*}
        C=\{(x,y) \in \mathbb{R}^{2n} : |y|x=\pm |x|y \},
    \end{align*}
    let $x=r\omega$, $y=\rho\omega$, where $r\in \mathbb{R}$, $\rho\in\mathbb{R}$, $\omega \in S^{n-1}$.
    We have
    \begin{align*}
        z^j&=(r+\sqrt{-1}\rho)\omega^j, \\
        dz^j&=dx^j+\sqrt{-1}dy^j \\
            &=\omega^j (dr+\sqrt{-1}d\rho)+(r+\sqrt{-1}\rho)d\omega^j.
    \end{align*}
    On the cone $C$ using $\omega \in S^{n-1}$,     
    \begin{align*}
        \sum_{j=1}^n z^jdz^j &= \sum_{j=1}^n (\omega^j)^2(r+\sqrt{-1}\rho)(dr+\sqrt{-1}d\rho) +
        \sum_{j=1}^n(r+\sqrt{-1}\rho) \omega^jd\omega^j \\
        &=(r+\sqrt{-1}\rho)(dr+\sqrt{-1}d\rho).
    \end{align*}
   
    Let $p\in \spt T$ and without loss of generality assume that $x^1(p)\neq 0$. 
    Consider the rotation vector fields in Lemma\ref{basis}.
    \begin{align*}
        X_j&= x^{j}\frac{\partial}{\partial x^1} -x^{1}\frac{\partial}{\partial x^j} 
              + y^{j}\frac{\partial}{\partial y^1} -y^{1}\frac{\partial}{\partial y^j}.
    \end{align*}
    We prove $\im (dr+\sqrt{-1}d\rho)(r+\sqrt{-1}\rho)^{n-1}=0$ locally.
    Notice that $dz^{1}(X_j)=z^j$ and $dz^{i}(X_j)=-z^1\delta_{j}^i$ for all $2\le i,j \le n$.
    Therefore, 
    \begin{align*}
        (dz)(-,X_2,\dots, X_n)&=-(-z^1)^{n-2}\sum_{j=1}^n z^jdz^j \\
        &=-(-\omega^1)^{n-2}(r+\sqrt{-1}\rho)^{n-1}(dr+\sqrt{-1}d\rho)
    \end{align*}
    Taking the imaginary part, we have
    \begin{align*}
        \im \{(r+\sqrt{-1}\rho)^{n-1}(dr+\sqrt{-1}d\rho)\}|_{M}=0. 
    \end{align*}

    Thus, $\varphi$ is locally constant on $\reg T$. Since $\varphi$ is continuous, $\varphi$ is constant on $\spt T$. 
    If $0\in \operatorname{spt(T)}$, then $T$ satisfy $\varphi=0$, which implies $T$ is a cone, therefore, $T=P+Q$.
    
    If $0\notin \spt T$, $T$ is smooth and $\varphi$ is a non-zero constant. This implies that $T$ is the Lagrangian catenoid.
\end{proof}

\section{Proof of main theorems}
If $T$ is a cone, then $T$ is a translation of $P+Q$. We assume that $T$ is not a cone. 
By theorem \ref{reg}, theorem \ref{inf} and Allard-Almgren theorem(\cite{Allard-Almgren}, \cite{Simon84}), $T$ satisfies the following:
\begin{itemize}
    \item the Hausdorff dimension of the singular set is at most $n-3$,
    \item $T$ is regular at infinity,
    \item the support of $T$ is connected.
\end{itemize}

We introduce calibrations according to $n$. Applying theorem \ref{calibration}, 
we prove $T$ is calibrated and each graph satisfies $u=O(|x|^{1-n})$.

\subsection{Proof of Theorem \ref{main1} $(n=2)$}
In this case, $T$ is a smooth oriented embedded minimal submanifold in $\mathbb R^4$ that is regular at infinity.

In $\mathbb{R}^4$, the symmetric angle condition \eqref{eq:symmetric angle} becomes 
\[
    \theta_1=\theta_2=\frac{\pi}{2}   
\]
We identify $\mathbb{R}^4$ with $\mathbb{C}^2$.
Choosing coordinates, we may assume that 
Two oriented planes are
\[
    P=\{(z^1,z^2)\in \mathbb{C}^2:z^1=0\}, \quad Q=\{(z^1,z^2)\in \mathbb{C}^2:z^2=0\}.
\]
Thus, $P+Q$ is calibrated by the K\"ahler form $\omega$, in the standard coordinates $z^j=x^j+iy^j$, for $j=1, 2$, we have
\[
    \omega = \sum_{j=1}^2 dx^j \wedge dy^j.
\]

By theorem \ref{calibration}, $T$ is calibrated by $\omega$.
Consider a holomorphic function $\mu=z^1z^2$ on $T$. The asymptotic condition implies that $\mu$ is bounded. Since $T$ has a finite density at infinity, $T$ is parabolic, and hence $\mu$ is constant. Up to a scaling and a rotation,
\[
    T=\{(z^1,z^2)\in \mathbb{C}^2:z^1z^2=1\}.
\]

\subsection{Proof of Theorem \ref{main1} $(n\ge 3)$}
We may assume that 
\begin{align*}
    P=\{(x,y) \in \mathbb{R}^{2n} : y=0\}, \quad
    Q=\{(x,y) \in \mathbb{R}^{2n} : y=\tan(\pi/n)x\}.
\end{align*}
We identify $\mathbb{R}^{2n}$ with $\mathbb{C}^n$. Since $P+Q$ is calibrated by $\re\Omega$, by theorem \ref{calibration}, $T$ is calibrated by $\re\Omega$, i.e. $T$ is special Lagrangian of phase $0$. Since $T$ is special Lagrangian, we may assume that each graph has the decay   
\begin{align*}
    u(x)=c_3\sum_{i=1}^n x|x|^{-n}+O(|x|^{-n}).
\end{align*}

    We work with $P$ and assume that $P$ is the $x$-plane, 
    \begin{align*}
        P=\{(x,y)\in \mathbb{R}^{2n}: y=0\}.
    \end{align*}

    If $P$ is not the $x$-plane, $P=\{x=\lambda y\}$ for some $\lambda \in \mathbb{R}$ because $P$ is invariant under $SO(n)$-diagonal action. We may choose a special unitary change of coordinates $\tilde{x^i}, \tilde{y^i}$ such that 
    $P=\{\tilde{y}=0\}$ and
    \begin{align*}
        \mu=x^iy^j-x^jy^i=\tilde{x}^i\tilde{y}^j-\tilde{x}^j\tilde{y}^i.
    \end{align*}
    
    Therefore, the same argument works for $Q$. Each $\mu=x^iy^j-x^jy^i$ has the asymptotic expansion $\mu=O(r^{1-n})$ and $\nabla \mu=O(r^{-n})$.

    Since $T$ is special Lagrangian, $\mu=x^iy^j-x^jy^i$ is a harmonic function on the regular set of $T$. Since the singular set is at most $n-3$, the standard covering argument implies that $\mu$ is weakly harmonic on $T$, i.e.
\begin{align*}
    \int_{\text{spt($T$)}} \nabla_T\zeta \cdot \nabla_T \mu=0 \text{ for all } \zeta\in C_{c}^1(\mathbb{C}^{n}).
\end{align*}    

    Since $\mu=O(r^{1-n})$, we have $\nabla_T \mu=0$ on $\reg T$. Thus, each $\mu=0$ on $\reg T$, therefore, $\mu=0$ on $\spt T$.
    
    This implies that the regular set of $T$ is $SO(n)$-invariant. Since the singular set is at most $n-3$, both $\spt T$ and $\sing T$ are $SO(n)$-invariant. Thus, we have $\spt T=\emptyset$ or $\spt T=\{0\}$. By lemma \ref{lemma:sLag}, in the first case, we obtain the Lagrangian catenoid. In the second case, we have $T=P+Q$.

\subsection{Proof of Theorem \ref{main2}}
As in the previous case, we can show that $M$ special Lagrangian. Since $H^1(M;\mathbb{R})=0$, $M$ is exact. By \cite[Theorem~1.1]{IJO2016}, $M$ is a Lawlor neck.

\bibliographystyle{amsplain}
\bibliography{references}

\end{document}